\documentclass[11pt]{article}
\usepackage[T1]{fontenc}
\usepackage{amsmath,amssymb,amsthm}
\usepackage[numbers,sort&compress]{natbib}
\usepackage[margin=1.05in]{geometry}
\usepackage[hidelinks]{hyperref}
\usepackage{fancyhdr}
\fancypagestyle{plain}{\fancyhf{}\fancyfoot[L]{\scriptsize General layout preview}\fancyfoot[C]{\thepage}}
\theoremstyle{plain}
\newtheorem{thm}{Theorem}[section]
\newtheorem{lemma}[thm]{Lemma}
\newtheorem{cor}[thm]{Corollary}
\newtheorem{problem}[thm]{Problem}
\theoremstyle{remark}
\newtheorem{rem}[thm]{Remark}
\numberwithin{equation}{section}
\title{A problem of Yang and Chen on weighted representation functions\thanks{This work was supported by the National Natural Science Foundation of China (Grant No. 12101009).}}
\author{Shuang-Shuang Li, Ya-Ting Xu and Xiao-Hui Yan\thanks{Corresponding author.}}
\date{}
\begin{document}
\maketitle
\begin{abstract}
Let $\mathbb N$ denote the set of nonnegative integers. For an integer $k>1$ and a set $A\subseteq\mathbb N$, let $R_{1,k}(A,n)$ denote the number of solutions of $n=a_1+ka_2$ with $a_1,a_2\in A$. For integers $k>1$ and $t\ge1$, Yang and Chen defined $f_k(t)$ to be the number of sets $A\subseteq\mathbb N$ for which
\[
R_{1,k}(A,n)=R_{1,k}(\mathbb N\setminus A,n)
\]
for all integers $n\ge t$, and asked whether $f_k(t)$ and $f_l(t)$ are eventually equal for any integers $k,l>1$. We prove that
\[
f_k(t)\asymp_k \frac{2^t}{t^{k/2}},
\]
which gives a negative answer to their problem.
\end{abstract}
\par
\medskip

\section{Introduction}

Let $\mathbb N$ be the set of all nonnegative integers. For positive integers $k_1,k_2$, a set $A\subseteq\mathbb N$ and $n\in\mathbb N$, let $R_{k_1,k_2}(A,n)$ be the number of solutions of
\[
n=k_1a_1+k_2a_2,\qquad a_1,a_2\in A.
\]
Weighted representation functions and partitions of $\mathbb N$ have been studied in a series of papers; see, for example, \cite{YangChen2012,Yang2014,Qu2016,LiShanYan2024,ChenDingLuZhang2024,YanShan2025}.

In 2012, Yang and Chen \cite{YangChen2012} determined all pairs $(k_1,k_2)$ for which there exists a set $A\subseteq\mathbb N$ such that
\[
R_{k_1,k_2}(A,n)=R_{k_1,k_2}(\mathbb N\setminus A,n)
\]
for all sufficiently large integers $n$. For integers $k>1$ and $t\ge1$, they defined $f_k(t)$ to be the number of sets $A\subseteq\mathbb N$ such that
\begin{equation}\label{eq:equality}
R_{1,k}(A,n)=R_{1,k}(\mathbb N\setminus A,n)
\end{equation}
for all integers $n\ge t$. They proved that $f_k(t)$ is finite and
\[
\lim_{t\to\infty}\frac{\log f_k(t)}{t}=\log2.
\]
Moreover, they posed the following problem.

\begin{problem}[Yang and Chen \cite{YangChen2012}]\label{prob:YC}
Is it true that for any integers $k,l>1$, there exists $t_0=t_0(k,l)$ such that
\[
f_k(t)=f_l(t)
\]
for all integers $t\ge t_0$?
\end{problem}


Our main result is the following general counting estimate.

\begin{thm}\label{thm:main}
Let $k>1$ and $c$ be fixed integers. Then, as $T\to\infty$,
\[
\#\left\{A\subseteq\mathbb N:
R_{1,k}(A,n)-R_{1,k}(\mathbb N\setminus A,n)=c
\text{ for all integers }n\ge T\right\}
\asymp_{k,c}\frac{2^T}{T^{k/2}}.
\]
\end{thm}

Taking $c=0$ gives the correct order of $f_k(t)$.

\begin{cor}\label{cor:f}
Let $k>1$ be fixed. Then, as $t\to\infty$,
\[
f_k(t)\asymp_k\frac{2^t}{t^{k/2}}.
\]
\end{cor}

If $1<k<l$, then Corollary \ref{cor:f} gives
\[
\frac{f_k(t)}{f_l(t)}
\gg_{k,l}t^{(l-k)/2}\longrightarrow\infty.
\]
Hence, $f_k(t)>f_l(t)$ for all sufficiently large integers $t$.

\begin{cor}\label{cor:problem}
If $1<k<l$, then $f_k(t)>f_l(t)$ for all sufficiently large integers $t$.
\end{cor}

Therefore, Problem \ref{prob:YC} has a negative answer.

\begin{rem}\label{rem:g}
Yan and Shan \cite{YanShan2025} defined $g_{k,m}(t)$, for integers $k>1$, $m\ge0$ and $t\ge0$, to be the number of sets $A\subseteq\mathbb N$ such that
\[
R_{1,k}(A,n)-R_{1,k}(\mathbb N\setminus A,n)=m+1
\]
for all integers $n\ge mk+t$. They also asked for the correct order of $g_{k,m}(t)$. In fact, taking $c=m+1$ and $T=mk+t$ in Theorem \ref{thm:main}, for each fixed $m$ and $k$ we obtain
\[
g_{k,m}(t)
\asymp_{k,m}
\frac{2^{mk+t}}{(mk+t)^{k/2}}
\asymp_{k,m}\frac{2^t}{t^{k/2}}.
\]
Hence
\[
\lim_{t\to\infty}\frac{\log g_{k,m}(t)}{t}=\log2.
\]
\end{rem}

\section{Lemmas}

For a fixed integer $k>1$ and a set $A\subseteq\mathbb N$, let
\[
\mathbf 1_A(j)=
\begin{cases}
1,&j\in A,\\
0,&j\notin A,
\end{cases}
\qquad
\varepsilon_j=\mathbf1_A(j)-\mathbf1_{\mathbb N\setminus A}(j)=
\begin{cases}
1,&j\in A,\\
-1,&j\notin A,
\end{cases}
\]
and
\[
D_A(n)=R_{1,k}(A,n)-R_{1,k}(\mathbb N\setminus A,n).
\]

\begin{lemma}\label{lem:basic}
If $n=kq+s$, where $q\ge0$ and $0\le s<k$, then
\begin{equation}\label{eq:basic}
2D_A(n)=\sum_{j=0}^{q}\varepsilon_j+\sum_{j=0}^{q}\varepsilon_{kj+s}.
\end{equation}
If $n\ge k$, then
\begin{equation}\label{eq:rec}
2D_A(n)-2D_A(n-k)=\varepsilon_{\lfloor n/k\rfloor}+\varepsilon_n.
\end{equation}
\end{lemma}

\begin{proof}
Clearly, for any integers $x,y\in\mathbb N$,
\[
2\{\mathbf 1_A(x)\mathbf 1_A(y)
-\mathbf 1_{\mathbb N\setminus A}(x)
 \mathbf 1_{\mathbb N\setminus A}(y)\}
=\varepsilon_x+\varepsilon_y.
\]
Since $n=kq+s$, the solutions of $n=a_1+ka_2$ with $a_1,a_2\in\mathbb{N}$ are
\[
(a_1,a_2)=(n-kj,j),\qquad 0\le j\le q.
\]
Hence,
\[
\begin{aligned}
2D_A(n)
&=2\sum_{j=0}^{q}\left(\mathbf 1_A(n-kj)\mathbf 1_A(j)
-\mathbf 1_{\mathbb N\setminus A}(n-kj)
 \mathbf 1_{\mathbb N\setminus A}(j)\right)\\
&=\sum_{j=0}^{q}(\varepsilon_{n-kj}+\varepsilon_j)\\
&=\sum_{j=0}^{q}\varepsilon_j
 +\sum_{j=0}^{q}\varepsilon_{k(q-j)+s}\\
&=\sum_{j=0}^{q}\varepsilon_j
 +\sum_{j=0}^{q}\varepsilon_{kj+s},
\end{aligned}
\]
which proves \eqref{eq:basic}. If $n\ge k$, then $q\ge1$ and
\[
\begin{aligned}
2D_A(n)-2D_A(n-k)
&=\left(\sum_{j=0}^{q}\varepsilon_j-
        \sum_{j=0}^{q-1}\varepsilon_j\right)
 +\left(\sum_{j=0}^{q}\varepsilon_{kj+s}-
        \sum_{j=0}^{q-1}\varepsilon_{kj+s}\right)\\
&=\varepsilon_q+\varepsilon_{kq+s}\\
&=\varepsilon_{\lfloor n/k\rfloor}+\varepsilon_n.
\end{aligned}
\]
This completes the proof of Lemma \ref{lem:basic}.
\end{proof}

\begin{lemma}\label{lem:boundary}
Let $c$ be an integer and $T$ be a positive integer. Then
\[
D_A(n)=c
\]
for all integers $n\ge T$ if and only if
\[
D_A(n)=c
\]
for all integers $T\le n\le T+k-1$, and
\begin{equation}\label{eq:tail}
\varepsilon_n=-\varepsilon_{\lfloor n/k\rfloor}
\qquad(n\ge T+k).
\end{equation}
\end{lemma}

\begin{proof}
Suppose that $D_A(n)=c$ for all $n\ge T$. Clearly,  $D_A(n)=c$ for all $T\le n\le T+k-1$. For $n\ge T+k$, by
\eqref{eq:rec},
\[
0=2D_A(n)-2D_A(n-k)
 =\varepsilon_{\lfloor n/k\rfloor}+\varepsilon_n,
\]
and hence \eqref{eq:tail} holds.

Conversely, suppose that $D_A(n)=c$ for all $T\le n\le T+k-1$ and
\eqref{eq:tail} holds. By \eqref{eq:rec} and \eqref{eq:tail} we have, for all integers $n\ge T+k$,
\[
2D_A(n)-2D_A(n-k)
=\varepsilon_{\lfloor n/k\rfloor}+\varepsilon_n=0.
\]
Thus, for all integers $n\ge T+k$,
\[
D_A(n)=D_A(n-k).
\]
For every $n\ge T$, there exists a nonnegative integer $r$ such that
\[
T\le n-rk\le T+k-1.
\]
It follows that
\[
D_A(n)=D_A(n-k)=\cdots=D_A(n-rk)=c.
\]
This completes the proof of Lemma \ref{lem:boundary}.
\end{proof}

\begin{lemma}\label{lem:binomial}
Let $M$ be a positive integer and $h$ be an integer. Then
\begin{align*}
&\#\left\{(\delta_1,\ldots,\delta_M)\in\{-1,1\}^M:
\delta_1+\cdots+\delta_M=h\right\}\\
&\qquad =
\begin{cases}
\binom{M}{(M+h)/2},&h\equiv M\pmod2\text{ and }|h|\le M,\\
0,&\text{otherwise}.
\end{cases}
\end{align*}
In particular,
\[
\#\left\{(\delta_1,\ldots,\delta_M)\in\{-1,1\}^M:
\delta_1+\cdots+\delta_M=h\right\}\ll \frac{2^M}{\sqrt M}.
\]
\end{lemma}

\begin{proof}
Let $\boldsymbol{\delta}=(\delta_1,\ldots,\delta_M)\in\{-1,1\}^M$ with
\[
\delta_1+\cdots+\delta_M=h.
\]
Then
\[
h=U_{\boldsymbol{\delta}}-(M-U_{\boldsymbol{\delta}})=2U_{\boldsymbol{\delta}}-M,
\]
where
\[
U_{\boldsymbol{\delta}}=\#\{i: \delta_i=1, 1\le i\le M\}.
\]
Thus
$h\equiv M\pmod2$, $|h|\le M$, and
\[
U_{\boldsymbol{\delta}}=\frac{M+h}{2}.
\]
Hence,
\begin{align*}
&\#\left\{(\delta_1,\ldots,\delta_M)\in\{-1,1\}^M:
\delta_1+\cdots+\delta_M=h\right\}\\
&\qquad =
\begin{cases}
\binom{M}{(M+h)/2},&h\equiv M\pmod2\text{ and }|h|\le M,\\
0,&\text{otherwise}.
\end{cases}
\end{align*}
By the Wallis inequalities (see \cite[p.~398, Eq.~(4)]{ChenQi2005}),
\[
\binom{M}{\lfloor M/2\rfloor}\asymp\frac{2^M}{\sqrt M}.
\]
It follows that if $h\equiv M\pmod2$ and $|h|\le M$, then
\[
\binom{M}{(M+h)/2}\le\binom{M}{\lfloor M/2\rfloor}\ll\frac{2^M}{\sqrt M}.
\]
This completes the proof of Lemma \ref{lem:binomial}.
\end{proof}
\begin{lemma}\label{lem:binomial2}
Let $H$ be a positive number, $M$ be a sufficiently large integer, $h$ be an integer with $h\equiv M\pmod2$. If $|h|\le H\sqrt M$, then
\[
\binom{M}{(M+h)/2}\asymp_H\frac{2^M}{\sqrt M}.
\]
\end{lemma}
\begin{proof}
For the upper bound, we have
\[
\binom{M}{(M+h)/2}\le\binom{M}{\lfloor M/2\rfloor}\ll\frac{2^M}{\sqrt M}.
\]
For the lower bound, let
\[
d=\frac{M+h}{2}-\left\lfloor\frac M2\right\rfloor.
\]
If $|h|\le H\sqrt M$, then
\[
|d|\le \frac{|h|}{2}+1\le \frac H2\sqrt M+1.
\]
For $M$ sufficiently large, we have $|d|\le M/4$ and $M/2-|d|>0$.

If $d=0$, then
\[
\binom{M}{(M+h)/2}=\binom{M}{\lfloor M/2\rfloor}\gg_H\frac{2^M}{\sqrt M}.
\]

If $d>0$, by
\[
\frac{\binom{M}{u+1}}{\binom{M}{u}}
=\frac{M-u}{u+1},
\]
then
\[
\frac{\binom{M}{(M+h)/2}}{\binom{M}{\lfloor M/2\rfloor}}
=
\prod_{i=0}^{d-1}\frac{M-\lfloor M/2\rfloor-i}{\lfloor M/2\rfloor+i+1}
\ge
\left(\frac{M/2-d}{M/2+d}\right)^d.
\]
Since
\[
\log\left(\frac{M/2-d}{M/2+d}\right)
=O\left(\frac dM\right),
\]
it follows that
\[
\log\left(
\left(\frac{M/2-d}{M/2+d}\right)^d
\right)
=O\left(\frac{d^2}{M}\right)=O_H(1).
\]
Hence
\[
\binom{M}{(M+h)/2}\gg_H\binom{M}{\lfloor M/2\rfloor}
\gg_H\frac{2^M}{\sqrt M}.
\]

Similarly, if $d<0$, by
\[
\frac{\binom{M}{u-1}}{\binom{M}{u}}
=\frac{u}{M-u+1},
\]
then
\[
\frac{\binom{M}{(M+h)/2}}{\binom{M}{\lfloor M/2\rfloor}}
=
\prod_{i=0}^{|d|-1}\frac{\lfloor M/2\rfloor-i}{M-\lfloor M/2\rfloor+i+1}
\ge
\left(\frac{M/2-|d|}{M/2+|d|}\right)^{|d|}\gg_H 1.
\]
Hence,
\[
\binom{M}{(M+h)/2}\gg_H\binom{M}{\lfloor M/2\rfloor}
\gg_H\frac{2^M}{\sqrt M}.
\]
This completes the proof of Lemma \ref{lem:binomial2}.
\end{proof}

Let $T$ be a positive integer. For every integer $0\le s<k$, let $n_s$ be the unique integer in $[T,T+k-1]$
with $n_s\equiv s\pmod k$, and let
\[
q_s=\frac{n_s-s}{k},
\qquad
Q=\max_{0\le s<k}q_s.
\]
Define
\[
P_s=\{x:Q<x\le n_s,\ x\equiv s\pmod k\},
\qquad
M_s=|P_s|.
\]
It is clear that $P_0,\cdots, P_{k-1}$ are pairwise disjoint.
\begin{lemma}\label{lem:blocks}
Let $0\le s<k$ be an integer. Then
\begin{equation}\label{eq:Ms}
M_s=\frac{k-1}{k^2}T+O_k(1).
\end{equation}
\end{lemma}

\begin{proof}
Let $Q<x\le n_s$ and $x\in P_s$. Write $x=s+kj$ with $j\ge0$. Then
\[
M_s
=\#\{j:(Q-s)/k<j\le q_s\}
=q_s-\left\lfloor\frac{Q-s}{k}\right\rfloor.
\]
Since
\[
0\le s<k,
\qquad
T\le n_s\le T+k-1,
\]
it follows that
\[
q_s=\frac{n_s-s}{k}=\frac Tk+O_k(1),
\qquad
Q=\max_{0\le s<k}q_s=\frac Tk+O_k(1).
\]
Hence,
\[
\begin{aligned}
M_s
&=q_s-\frac Qk+O_k(1)\\
&=\frac Tk-\frac{T}{k^2}+O_k(1)\\
&=\frac{k-1}{k^2}T+O_k(1).
\end{aligned}
\]
This completes the proof of Lemma \ref{lem:blocks}.
\end{proof}

Let
\begin{equation}\label{eq:C}
C=\{0,1,\ldots,T+k-1\}\setminus
\bigcup_{s=0}^{k-1}P_s.
\end{equation}
Since $P_0,\cdots, P_{k-1}$ are pairwise disjoint, it follows from Lemma \ref{lem:basic} that
\begin{equation}\label{eq:boundaryeq}
D_A(n_s)=c
\quad\Longleftrightarrow\quad
\sum_{j=0}^{q_s}\varepsilon_j+
\sum_{j=0}^{q_s}\varepsilon_{kj+s}=2c
\quad\Longleftrightarrow\quad
\sum_{x\in C}a_{s,x}\varepsilon_x+\sum_{x\in P_s}\varepsilon_x=2c,
\end{equation}
where
\[
a_{s,x}=\mathbf 1_{\{0,\cdots,q_s\}}(x)+\mathbf 1_{\{kj+s: 0\le j\le q_s\}}(x).
\]
It follows from Lemma \ref{lem:boundary} that in order to complete the proof of Theorem \ref{thm:main}, it suffices to determine the number of sign vectors
\[
(\varepsilon_0,\ldots,\varepsilon_{T+k-1})\in\{-1,1\}^{T+k}
\]
satisfying
\[
\sum_{x\in C}a_{s,x}\varepsilon_x+
\sum_{x\in P_s}\varepsilon_x=2c,
\qquad 0\le s<k.
\]

\begin{lemma}\label{lem:goodcore}
Let $k>1$ and $T\ge k$ be integers, and let $C$ be the set defined by \eqref{eq:C}. Then
\[
 \#\left\{(\varepsilon_x)_{x\in C}\in\{1,-1\}^{|C|}:
 \left|\sum_{x\in C}a_{s,x}\varepsilon_x\right|\le4\sqrt{kT}\text{ for every }0\le s<k\right\}\ge 2^{|C|-1}.
\]
\end{lemma}

\begin{proof}
Note that  $(\varepsilon_x)_{x\in C}$ denotes a choice of signs \(\varepsilon_x\in\{-1,1\}\) for all \(x\in C\). There are $2^{|C|}$ possible choices of the signs $(\varepsilon_x)_{x\in C}$, since each $\varepsilon_x$ is either $1$ or $-1$.
For $x,y\in C$, if $x=y$, then
\[
\sum_{(\varepsilon_z)_{z\in C}}
\varepsilon_x\varepsilon_y
=\sum_{(\varepsilon_z)_{z\in C}}1=2^{|C|},
\]
if $x\neq y$, then
\[
\sum_{(\varepsilon_z)_{z\in C}}
\varepsilon_x\varepsilon_y
=\sum_{(\varepsilon_z)_{z\in C\setminus\{x,y\}}}\left(\sum_{\varepsilon_x, \varepsilon_y}\varepsilon_x\varepsilon_y\right)
=\sum_{(\varepsilon_z)_{z\in C\setminus\{x,y\}}}\left(1-1-1+1\right)
=0,
\]
that is,
\[
\sum_{(\varepsilon_z)_{z\in C}}
\varepsilon_x\varepsilon_y
=
\begin{cases}
2^{|C|},&x=y,\\
0,&x\ne y.
\end{cases}
\]
Since $|a_{s,x}|\le2$ and $T\ge k$, it follows that $|C|\le 2T$ and, for each fixed $s$,
\[
\begin{aligned}
\sum_{(\varepsilon_z)_{z\in C}}\left(\sum_{x\in C}a_{s,x}\varepsilon_x\right)^2
&=\sum_{x,y\in C}a_{s,x}a_{s,y}
  \sum_{(\varepsilon_z)_{z\in C}}\varepsilon_x\varepsilon_y\\
&=2^{|C|}\sum_{x\in C}a_{s,x}^2\\
&\le4|C|2^{|C|}\\
&\le8T2^{|C|}.
\end{aligned}
\]
Therefore
\[
16kT\,
\#\left\{(\varepsilon_x)_{x\in C}:\left|\sum_{x\in C}a_{s,x}\varepsilon_x\right|>4\sqrt{kT}\right\}
<\sum_{(\varepsilon_z)_{z\in C}}\left(\sum_{x\in C}a_{s,x}\varepsilon_x\right)^2
\le8T2^{|C|},
\]
and so
\[
\#\left\{(\varepsilon_x)_{x\in C}:\left|\sum_{x\in C}a_{s,x}\varepsilon_x\right|>4\sqrt{kT}\right\}
\le\frac{2^{|C|}}{2k}.
\]
Thus
\[
\begin{aligned}
&\#\left\{(\varepsilon_x)_{x\in C}:
 \left|\sum_{x\in C}a_{s,x}\varepsilon_x\right|\le4\sqrt{kT}\text{ for every }0\le s<k\right\}\\
&\qquad\ge
2^{|C|}-\sum_{s=0}^{k-1}
\#\left\{(\varepsilon_x)_{x\in C}:\left|\sum_{x\in C}a_{s,x}\varepsilon_x\right|>4\sqrt{kT}\right\}\\
&\qquad\ge
2^{|C|}-k\frac{2^{|C|}}{2k}
=2^{|C|-1}.
\end{aligned}
\]
This completes the proof of Lemma \ref{lem:goodcore}.
\end{proof}

\section{Proof of Theorem \ref{thm:main}}

\begin{proof}[Proof of Theorem \ref{thm:main}]
Let $c$ be a fixed integer and let $T$ be a sufficiently large integer, and let $C$ be the set defined by \eqref{eq:C}. For every integer $0\le s<k$, let
\[
 h_s=2c-\sum_{x\in C}a_{s,x}\varepsilon_x.
\]
By Lemma
\ref{lem:boundary} and \eqref{eq:boundaryeq},
it suffices to determine the number of sign vectors
\[
(\varepsilon_0,\ldots,\varepsilon_{T+k-1})\in\{-1,1\}^{T+k}
\]
satisfying
\[
\sum_{x\in C}a_{s,x}\varepsilon_x+
\sum_{x\in P_s}\varepsilon_x=2c,
\qquad 0\le s<k,
\]
that is,
\begin{equation}\label{eq:target}
\sum_{x\in P_s}\varepsilon_x
=h_s,\qquad 0\le s<k.
\end{equation}

We first prove the upper bound. Since $M_s=|P_s|$, it follows from Lemma \ref{lem:binomial} with $h=h_s$ that
\[
\begin{aligned}
\#\left\{(\varepsilon_x)_{x\in P_s}:
\sum_{x\in P_s}\varepsilon_x
=h_s\right\}
\ll \frac{2^{M_s}}{\sqrt{M_s}}.
\end{aligned}
\]
Since $P_0,\ldots,P_{k-1}$ are pairwise disjoint, the number of choices of $\varepsilon_x$ with $x\in C$ is $2^{|C|}$, and
\[
\{0,1,\cdots,T+k-1\}=C\cup P_0\cup\cdots \cup P_{k-1},\quad |P_s|=M_s,
\]
it follows that $|C|+M_0+\cdots+M_{k-1}=T+k$ and
\[
\begin{aligned}
&\#\{A\subseteq\mathbb N:D_A(n)=c\text{ for all integers }n\ge T\}\\
&\qquad\ll_k
2^{|C|}\prod_{s=0}^{k-1}\frac{2^{M_s}}{\sqrt{M_s}}
=
\frac{2^{|C|+M_0+\cdots+M_{k-1}}}
{\sqrt{M_0M_1\cdots M_{k-1}}}
=\frac{2^{T+k}}
{\sqrt{M_0M_1\cdots M_{k-1}}}.
\end{aligned}
\]
By \eqref{eq:Ms}, we have
\[
M_s=\frac{k-1}{k^2}T+O_k(1),
\qquad 0\le s<k,
\]
and so
\[
M_0M_1\cdots M_{k-1}\gg_k T^k.
\]
It follows that
\[
\#\{A\subseteq\mathbb N:D_A(n)=c\text{ for all integers }n\ge T\}
\ll_k\frac{2^T}{T^{k/2}}.
\]

We next prove the lower bound. Since $|P_s|=M_s$, in order to determine the number of sign vectors satisfying \eqref{eq:target} by using Lemma \ref{lem:binomial}, we must verify $h_s\equiv M_s\pmod2$ and $|h_s|\le M_s$. By \eqref{eq:boundaryeq}, we know that the sum of the coefficients in the
corresponding boundary equation is $2(q_s+1)$. Hence
\[
\sum_{x\in C}a_{s,x}+M_s=2(q_s+1).
\]
It follows from \eqref{eq:target} and $\varepsilon_x\in \{1,-1\}$ that
\begin{equation}\label{eq:parity}
h_s=2c-\sum_{x\in C}a_{s,x}\varepsilon_x
\equiv-\sum_{x\in C}a_{s,x}
\equiv M_s\pmod2.
\end{equation}
By Lemma \ref{lem:goodcore}, there are at
least $2^{|C|-1}$ choices of $(\varepsilon_x)_{x\in C}$ such that
\[
\left|\sum_{x\in C}a_{s,x}\varepsilon_x\right|
\le4\sqrt{kT}
\]
for every $0\le s<k$.

It follows that
\[
|h_s|=\left|2c-\sum_{x\in C}a_{s,x}\varepsilon_x\right|
\le2|c|+4\sqrt{kT}
\ll_{k,c}\sqrt T.
\]
By \eqref{eq:Ms}, for all sufficiently large integers $T$, we have
\begin{equation}\label{eq:centralrange}
|h_s|=\left|2c-\sum_{x\in C}a_{s,x}\varepsilon_x\right|
\ll_{k,c}\sqrt{M_s}\le M_s.
\end{equation}
Thus, by \eqref{eq:parity}, \eqref{eq:centralrange}, Lemma
\ref{lem:binomial} and Lemma \ref{lem:binomial2},
\[
\#\left\{(\varepsilon_x)_{x\in P_s}:
\sum_{x\in P_s}\varepsilon_x
=h_s\right\}
=
\binom{M_s}{\displaystyle
\left(M_s+h_s\right)/2}
\gg_{k,c}\frac{2^{M_s}}{\sqrt{M_s}}.
\]
Since $P_0,\ldots,P_{k-1}$ are pairwise disjoint, the number of choices of $\varepsilon_x$ with $x\in C$ satisfying \eqref{eq:centralrange} is at least $2^{|C|-1}$, and
\[
\{0,1,\cdots,T+k-1\}=C\cup P_0\cup\cdots \cup P_{k-1},\quad |P_s|=M_s,
\]
it follows from \eqref{eq:Ms} that $|C|+M_0+\cdots+M_{k-1}=T+k$ and
\[
\#\{A\subseteq\mathbb N:D_A(n)=c\text{ for all }n\ge T\}
\gg_{k,c}
2^{|C|-1}\prod_{s=0}^{k-1}\frac{2^{M_s}}{\sqrt{M_s}}
\gg_{k,c}\frac{2^T}{T^{k/2}}.
\]

Hence,
\[
\#\left\{A\subseteq\mathbb N:
D_A(n)=c\text{ for all }n\ge T\right\}
\asymp_{k,c}\frac{2^T}{T^{k/2}}.
\]
This completes the proof of Theorem \ref{thm:main}.
\end{proof}

\bigskip
\noindent\textsc{Shuang-Shuang Li}, Office of Scientific Research,\\
Anhui Normal University, Wuhu 241002, P. R. China\\
\textit{E-mail:} \texttt{ddlshuang@163.com}\par
\medskip
\noindent\textsc{Ya-Ting Xu}, School of Mathematics and Statistics,\\
Anhui Normal University, Wuhu 241002, P. R. China\\
\textit{E-mail:} \texttt{980181682@qq.com}\par
\medskip
\noindent\textsc{Xiao-Hui Yan}, School of Mathematics and Statistics,\\
Anhui Normal University, Wuhu 241002, P. R. China\\
\textit{E-mail:} \texttt{yanxiaohui\_1992@163.com}

\end{document}